\documentclass{amsart}
\usepackage{amsfonts}
\usepackage{amsmath,amssymb}
\usepackage{amsthm}
\usepackage{amscd}
\usepackage{graphics}
\usepackage{graphicx}

\theoremstyle{remark}{
\newtheorem{Def}{{\rm Definition}}
\newtheorem{Ex}{{\rm Example}}
\newtheorem{Rem}{{\rm Remark}}
\newtheorem{Prob}{{\rm Problem}}

}
\theoremstyle{plain}
{

\newtheorem{Thm}{Theorem}

}

\begin{document}
\title[Graphs on Reeb spaces of functions being analytic on dense subsets]{Reeb spaces of functions being analytic on dense subsets and their graph structures}
\author{Naoki kitazawa}

\address{Osaka Central Advanced Mathematical Institute (OCAMI) \\
3-3-138 Sugimoto, Sumiyoshi-ku Osaka 558-8585
TEL: +81-6-6605-3103
}
\email{naokikitazawa.formath@gmail.com}
\urladdr{https://naokikitazawa.github.io/NaokiKitazawa.html}
\maketitle
\begin{abstract}

{\it Reeb spaces} of real-valued functions on manifolds are the spaces of all connected components (contours) of level sets and endowed with the natural quotient topology. They have been fundamental and strong tools in investigating manifolds via smooth functions with mild critical points since the birth of fundamental theory of Morse functions in the 20th century. 

We are concerned with topologies and combinatorics of them. Following an explicit note on explicit Reeb spaces of explicit functions which are real analytic (on dense sets) and seem to be simplest and most fundamental, edited by the author himself.

We investigate other construction of examples of such functions and their Reeb spaces. Reeb spaces are naturally graphs in considerable cases and as another work, we also discuss natural definitions of vertices for them.

\end{abstract}
\section{Introduction.}
\label{sec:1}
In the present section, which is for the introduction, \cite{kitazawa8}, a preprint of the author, is respected.

The {\it Reeb space} $R_c$ of a (continuous) map $c:X \rightarrow \mathbb{R}$ on a topological space $X$ is the space of all connected components ({\it contours}) of all preimages ({\it level sets}) $c^{-1}(q)$. Such objects have been fundamental and strong tools in investigating the differentiable manifold $X$ via the function $c$ since the establishment of fundamental theory of so-called Morse functions (\cite{reeb}).

We explain the notion rigorously. The equivalence relation ${\sim}_c$ on $X$ can be defined by the rule that the relation $p_1 {\sim}_c p_2$ holds if and only if $p_1$ and $p_2$ are in a same contour. The {\it Reeb space} $R_c$ of $c$ is the quotient space $R_c:=X/{\sim}_c$. The quotient map $q_c:X \rightarrow R_c$ is defined and the unique continuous map $\bar{c}:R_c \rightarrow Y$ with the relation $c=\bar{c} \circ q_c$ is obtained. In the case $X$ is a differentiable manifold, a contour of $c$ containing critical points is a {\it critical} contour of $c$. A contour of $c$ which is not critical is a {\it regular} contour of $c$. We also call a point of $R_c$ a contour of $c$.

In some nice situation, $R_c$ is homeomorphic to a graph. In \cite{izar}, Izar has been shown in the case of a Morse function on a compact manifold in \cite{izar}. Martinez-Alfaro, Meza-Sarmiento, and Oliveira, have presented a related study in \cite{martinezalfaromezasarmientooliveira}, considering the {\it Morse-Bott} functions. Gelbukh and Saeki have presented general theory on this in \cite{gelbukh, saeki2, saeki3}, for example. Especially, \cite[Theorem 7.5]{gelbukh}, \cite[Theorem 3.1]{saeki2}, and \cite[Theorems 2.1 and 2.8 and "2 Reeb space and its graph structure"]{saeki3} are on topologies and graph structures of Reeb spaces. In the case of the smooth manifold $X$ with no boundary, $R_c$ is, in considerable case, a graph whose vertex set consists of all critical contours of $c$. This is the well-known definition of the {\it Reeb graph} of $c$. However, as in \cite[Theorem 7.5]{gelbukh} and \cite["2 Reeb space and its graph structure"]{saeki3}, only the fact that a Reeb space is homeomorphic to a graph is important. In a preprint \cite{kitazawa3}, the author has discussed new definitions of graph structures on Reeb spaces and this is important in some of our main result (Theorems \ref{thm:2} and \ref{thm:3}).

What follows is our main problem, which is same as \cite{kitazawa8}.
\begin{Prob}
\label{prob:1}
Can we construct a nice smooth function whose space is homeomorphic to a given topological space?
\end{Prob}

This is incredibly new. This has been essentially started by Sharko in 2006 (\cite{sharko}). This is on smooth functions on closed surfaces whose critical points $p$ are of the form $c(z)={\Re} z^{l}+t_p$ and whose Reeb graphs are given graphs: $z \in \mathbb{C}$, with $\mathbb{C}$ being the space of all complex numbers, $l>1$ being an integer, $t_p \in \mathbb{R}$, and ${\Re} s$ denotes the real part of $s \in \mathbb{C}$. This is extended to arbitrary finite graphs in \cite{masumotosaeki}. In \cite{michalak}, finite graphs of a certain class are considered and Morse functions on closed manifolds with prescribed Reeb graphs and regular contours being spheres are reconstructed. 
The author has contributed to this story on reconstruction of explicit smooth functions whose Reeb graphs are isomorphic to given graphs by, \cite{kitazawa1, kitazawa2, kitazawa3, kitazawa4}, by respecting topologies of level sets in addition, considering smooth functions on non-compact manifolds and so-called {\it non-proper} functions, and presenting real algebraic construction, for example.

Especially, real algebraic construction is pioneered by the author in 2023 \cite{kitazawa3}. Constructing real algebraic or real analytic functions is difficult since we cannot use objects and methods in differentiable (smooth) category. In the differentiable situation, we locally construct local functions corresponding to vertices and product bundles corresponding to edges and glue them together to have smooth function whose Reeb graph is desired. This is remarked in \cite{kitazawa7} and extended in \cite{kitazawa6} for example, and developing now mainly due to the author. 

We present one of original idea of the author. We consider a region in the $n$-dimensional Euclidean space ({\it real affine space}) ${\mathbb{R}}^n$ surrounded by real algebraic hypersurfaces, construct a real algebraic map onto the region, generalizing the canonical projection of the unit sphere $S^{m}:=\{x=(x_1, \cdots x_{m+1}) \in {\mathbb{R}}^{m+1} \mid {\Sigma}_{j=1}^{m+1} {x_j}^2=1\}$ onto the unit disk $D^n:=\{x=(x_1, \cdots x_{n}) \in {\mathbb{R}}^{n} \mid {\Sigma}_{j=1}^{n} {x_j}^2 \leq 1\}$. This is, in terminologies of the singularity theory of differentiable maps and applications to differentiable topology of manifolds, constructing so-called special generic maps in the real algebraic situation. By composing the canonical projection, we have a desired function. For special generic maps, see \cite{saeki1} as a fundamental differential topological study. Recent preprint \cite{kitazawa9} is a related study where we do not assume related knowledge. The article \cite{bodinpopescupampusorea} is also a related elementary real algebraic study and the author has been also helped by this study.

Recently, a real analytic study which is essentially different from real algebraic studies is presented by the author, as \cite{kitazawa8} and an explicit analytic function on a non-compact manifold and a smooth function on closed manifold which is real analytic on an open and dense subset are presented and their Reeb spaces are investigated. They are explicit examples of Reeb spaces which are not finite graphs. Although such cases have been presented in \cite{gelbukh, saeki1, saeki2}, different from them, explicit representation via elementary functions are explicitly presented first.

In the present paper, we give new examples with new methods. We introduce some additional notation. Let ${\pi}_{k_1+k_2,k_2}:{\mathbb{R}}^{k_1+k_2} \rightarrow {\mathbb{R}}^{k_1}$ denote the canonical projection ${\pi}_{k_1+k_2,k_2}(x):=x_1$ ($x:=(x_1,x_2) \in {\mathbb{R}}^{k_1} \times {\mathbb{R}}^{k_2}={\mathbb{R}}^{k_1+k_2}$ with $k_1, k_2>0$). We use $0 \in {\mathbb{R}}^k$ for the origin of ${\mathbb{R}}^k$. A {\it graph} means a $1$-dimensional CW complex the closure of each 1-cell ({\it edge}) of which is homeomorphic to $D^1$. A proper map $c:X \rightarrow Y$ between topological spaces $X$ and $Y$ means a map whose preimage $c^{-1}(K)$ is compact for any compact subset $K \subset X$. Theorem \ref{thm:1} is our new result. 
\begin{Thm}
	\label{thm:1}
	
	There exist a connected infinite graph $G_0$ and an $m$-dimensional smooth connected submanifold $X_{m} \subset {\mathbb{R}}^{m+2}$ which is non-compact and represented as the zero set ${e_m}^{-1}(0)$ of some smooth map $e_m:{\mathbb{R}}^{m+2} \rightarrow {\mathbb{R}}^2$ with a function $c_{m}:X_{m} \rightarrow \mathbb{R}$ enjoying the following properties, where $m>1$ is an arbitrary integer greater than $1$.
	\begin{enumerate}
\item \label{thm:1.1} The restriction of $e_m$ to an open set of the form $\{t \mid -a<t<a\} \times {\mathbb{R}}^{m+1} \subset {\mathbb{R}}^{m+2}$ {\rm (}$a>0${\rm )} is a real polynomial map and $X_{m}$ is also in $\{t \mid -a<t<a\} \times {\mathbb{R}}^{m+1}$.
		\item \label{thm:1.2} The function $c_{m}$ is the restriction of ${\pi}_{m+1,1}$ to $X_{m}$. 
		\item \label{thm:1.3} The Reeb space $R_{c_{m}}$ is also regarded as the Reeb graph of $c$ and two graphs $G_0$ and $R_{c_{m}}$ are isomorphic. 
		\item \label{thm:1.4} The map $q_{c_{m}}:X_m \rightarrow R_{c_{m}}$ is proper.
	\end{enumerate}
\end{Thm}
Compare this to \cite[Theorem 1]{kitazawa8}.  
In the second section, we prove Theorem \ref{thm:1}. In the proof, methods in \cite{kitazawa3, kitazawa6, kitazawa8} are respected, where we do not assume them. In the third section, we present another new work. For this, here, we introduce a new definition of graph structures on Reeb spaces in a self-contained way, reviewing and following the preprint \cite{kitazawa5}. We choose a subset $Z$ in the space $X$ of the domain whose complementary set $X-Z$ is dense in $X$. 
We call the pair $(R_c,q_c(Z))$ the {\it $(R_c,Z)$-space}. If the $(R_c,Z)$-space is regarded as a graph whose vertex set is $q_c(Z)$, then we call the graph the {\it $(R_c,Z)$-graph} and let $R_{c,Z}$ denote the $(R_c,Z)$-graph. Under this setting, we answer to a kind of variants of Problem \ref{prob:1} by Theorem \ref{thm:2}. This also comes from our observations on examples presented as \cite[Theorem 2]{kitazawa8}. This is presented more precisely in the third section.

\begin{Thm}
\label{thm:2}
Each $S^m$ with $m \geq 2$ is diffeomorphic to a smooth submanifold $Y_{m,i}$ represented as the zero set ${e_{Y,m}}^{-1}(0)$ of some smooth map $e_{Y,m}:{\mathbb{R}}^{m+2} \rightarrow {\mathbb{R}}^2$ which is real analytic outside some set $Z_{m+2,i}$ with ${\mathbb{R}}^{m+2}-Z_{m+2,i}$ being dense in ${\mathbb{R}}^{m+2}$ and which is not real analytic at any point of $Z_{m+2,i}$ and the corresponding case in the following holds for $i=1,2,3$.
\begin{enumerate}
\setcounter{enumi}{4}
\item \label{thm:2.1} The Reeb space $R_{c_{Y,m,1}}$ of the restriction $c_{Y,m,1}$ of ${\pi}_{m+2,1}$ to $Y_{m,1}$ is not homeomorphic to any graph.
\item \label{thm:2.2} The Reeb space $R_{c_{Y,m,2}}$ of the restriction $c_{Y,m,2}$ of ${\pi}_{m+2,1}$ to $Y_{m,2}$ is regarded as the Reeb graph of $c_{Y,m,2}$ and the graph $R_{c_{Y,m,2},Y_{m,2} \bigcap Z_{m+2,2}}$ is not defined.
\item \label{thm:2.3} The Reeb space $R_{c_{Y,m,3}}$ of the restriction $c_{Y,m,3}$ of ${\pi}_{m+2,1}$ to $Y_{m,3}$ is regarded as the Reeb graph of $c_{Y,m,3}$ and the graph $R_{c_{Y,m,3},Y_{m,3} \bigcap Z_{m+2,3}}$ is defined.

\end{enumerate}

\end{Thm}
\section{On Theorem \ref{thm:1}.}
Let $\mathbb{Z}$ denote the set of all integers.
\begin{proof}[A proof of Theorem \ref{thm:1}]

We first define a smooth map $e_m:{\mathbb{R}}^m \rightarrow \mathbb{R}$, an open set $\{t \mid -a<t<a\} \times {\mathbb{R}}^m={\mathbb{R}}^{m+1}$, a manifold $X_m$, and a function $c_m:X_m \rightarrow \mathbb{R}$, in STEP 1-1. In this step, we also prove the property (\ref{thm:1.1}). In STEP 1-2, we prove the remaining properties (\ref{thm:1.2}), (\ref{thm:1.3}), and (\ref{thm:1.4}). 

\ \\
STEP 1-1 A smooth map $e_m:{\mathbb{R}}^m \rightarrow \mathbb{R}$, an open set $\{t \mid -a<t<a\} \times {\mathbb{R}}^m={\mathbb{R}}^{m+1}$, a manifold $X_m$, and a function $c_m:X_m \rightarrow \mathbb{R}$.

Let $S_{1}:={\bigcup}_{j \in \mathbb{Z}} \{(x_1,x_2)\mid  x_1-{(x_2-4j)}^2+\frac{1}{2}=0\}$. Let $S_{2}:={\bigcup}_{j \in \mathbb{Z}} \{(x_1,x_2)\mid  x_1+{(x_2-4j-2)}^2-\frac{1}{2}=0\}$. We can see $S_1$ and $S_2$ are disjoint. They are unions of parabolas. 
Let $S_{i+3}:=\{(2i-1,x_2) \mid x_2 \in \mathbb{R}\}$ for $i=0,1$.
Let $D_S:=\{(x_1,x_2) \mid  x_1-{(x_2-4j)}^2+\frac{1}{2}<0\ {\rm (}j \in \mathbb{Z}{\rm )}, x_1+{(x_2-4j-2)}^2-\frac{1}{2}>0\ {\rm (}j \in \mathbb{Z}{\rm )},-1<x_1<1\} \subset {\mathbb{R}}^2$, which is a region surrounded by these curves $S_i$ ($i=1,2,3,4$).
The closure $\overline{D_S}:=\{(x_1,x_2) \mid x_1-{(x_2-4j)}^2+\frac{1}{2} \leq 0\ {\rm (}j \in \mathbb{Z}{\rm )}, x_1+{(x_2-4j-2)}^2-\frac{1}{2} \geq 0\ {\rm (}j \in \mathbb{Z}{\rm )},-1 \leq x_1 \leq 1\} \subset {\mathbb{R}}^2$ is also important.
The set $S_i \bigcap \overline{D_S}$ ($i=1,2$) is the disjoint union of countably many connected curves in parabolas. Each of the set is extended to a curve satisfying the relation $\{(c_{S_i}(x_2),x_2) \mid x_2 \in \mathbb{R}\} \bigcap \overline{D_S}=S_i \bigcap \overline{D_S}$ with the following.
\begin{itemize}
\item $c_{S_i}$ is a smooth function such that on ${\bigcup}_{j \in \mathbb{Z}} \{x_2 \mid 4j-2i-\frac{1}{2}<x_2<4j-2i+\frac{5}{2}\}$, it is a real polynomial function of degree $2$ for parabolas and that outside the set ${\bigcup}_{j \in \mathbb{Z}} \{x_2 \mid 4j-2i-\frac{1}{2}<x_2<4j-2i+\frac{5}{2}\}$ in $\mathbb{R}$, the absolute values of the values of $c_{s_i}$ are always greater that ${(\frac{3}{2})}^2-\frac{1}{2}=\frac{7}{4}$.
\item The following are equivalent.
\begin{itemize}
\item $-1 \leq c_{S_i}(x_2) \leq 1$.

\item $4j-2(i-1)-\frac{\sqrt{3}}{2} \leq x_2 \leq 4j-2(i-1)+\frac{\sqrt{3}}{2}$, $j \in \mathbb{Z}$.
\end{itemize}
\end{itemize}
Let $m_1$ and $m_2$ be positive integers satisfying $m=m_1+m_2$. 
We define $X_m:=\{(x_1,x_2,(y_{1,j})_{j=1}^{m_1},(y_{2,j})_{j=1}^{m_2}) \mid (c_{S_1}(x_2)-x_1)(x_1-c_{S_2}(x_2))-{\Sigma}_{j=1}^{m_1} {y_{1,j}}^2=0, (x_1+1)(1-x_1)-{\Sigma}_{j=1}^{m_1} {y_{2,j}}^2=0, -1<x_1<1\}$ and a map $e_m:{\mathbb{R}}^{m+2} \rightarrow {\mathbb{R}}^2$ by $e_m(x_1,x_2,(y_{1,j})_{j=1}^{m_1},(y_{2,j})_{j=1}^{m_2}):=(c_{S_1}(x_2)-x_1)(x_1-c_{S_2}(x_2))-{\Sigma}_{j=1}^{m_1} {y_{1,j}}^2, (x_1+1)(1-x_1)-{\Sigma}_{j=1}^{m_1} {y_{2,j}}^2)$. Let $e_{m,1}$ and $e_{m,2}$ denote the 1st and the 2nd components of the map $e_m$, respectively.  

For this, we apply an argument in \cite[A proof of Theorem 1]{kitazawa8}. We also respect \cite{kitazawa3, kitazawa6} for this. We discuss our proof self-contained way.

We prove that $X_m$ is a smooth submanifold of ${\mathbb{R}}^{m+2}$ by the implicit function. 
First, by the definition, $X_m$ is mapped onto $\overline{D_S}$ by ${\pi}_{m+2,2}$.

We consider the following three cases to see that $X_m$ is an $m$-dimensional smooth manifold and $X_m={e_m}^{-1}(0)$. \\
\ \\
Case 1-1-1. At a point $(x_{0,1},x_{0,2},y_0) \in X_m$ such that $(x_{0,1},x_{0,2}) \in D_S$. \\
There, the value of the partial derivative of $e_{m_1}$ by some $y_{1,j}$ is not zero and that of $e_{m_2}$ by some $y_{2,j}$ is not zero. 
There, the value of the partial derivative of $e_{m_1}$ by some $y_{2,j}$ is $0$ and that of $e_{m_2}$ by some $y_{1,j}$ is $0$. We have shown that at the point $(x_{0,1},x_{0,2},y_0) \in X_m$ such that $(x_{0,1},x_{0,2}) \in D_S$, the rank of the differential of $e_m$ is $2$. \\
\ \\
Case 1-1-2 At a point $(x_{0,1},x_{0,2},y_0) \in X_m$ such that $(x_{0,1},x_{0,2}) \in \overline{D_S}-D_S$ and that $(x_{0,1},x_{0,2})$ is in exactly one curve from $S_i$ ($i=1,2,3,4$). \\
Let $(x_{0,1},x_{0,2}) \in S_1 \bigcup S_2$. The value of the partial derivative of the function $e_{m,1}$ by $x_1$ or $x_2$ is not $0$ by the smoothness of the curves $S_1$ and $S_2$, there. The value of the partial derivative of the function $e_{m,1}$ by $y_{i,j}$ is $0$, there. The value of the partial derivative of the function $e_{m,2}$ by some $y_{2,j}$ is not $0$, there. At a point $(x_{0,1},x_{0,2},y_0) \in X_m$ such that $(x_{0,1},x_{0,2}) \in \overline{D_S}-D_S$ and that $(x_{0,1},x_{0,2})$ is in exactly one curve from $S_1$ and $S_2$ the rank of the differential of $e_m$ is $2$. 

In the case $(x_{0,1},x_{0,2}) \in S_3 \bigcup S_4$, we can argue similarly, by the symmetry. \\
\ \\
Case 1-1-3 At a point $(x_{0,1},x_{0,2},y_0) \in X_m$ such that $(x_{0,1},x_{0,2}) \in \overline{D_S}-D_S$ and that $(x_{0,1},x_{0,2})$ is in exactly two curves from $S_i$ ($i=1,2,3,4$). \\

There, the value of the partial derivative of $e_{m_1}$ by some $x_a$ is not zero and that of $e_{m_2}$ by some $x_a$ is not zero. There, the value of the partial derivative of $e_{m_1}$ by any $y_{a,j}$ is $0$ and that of $e_{m_2}$ by any $y_{a,j}$ is $0$. 
By our setting, we can see that in $\overline{D_S}$, at most two curves from $S_i$ ($i=1,2,3,4$) intersect and that they intersect in transversal ways. In other words, the corresponding tangent vectors which are not the zero vectors at each point of the intersections are mutually independent and for normal vectors there, a similar fact holds.

From this, we can see that at the point $(x_{0,1},x_{0,2},y_0) \in X_m$ such that $(x_{0,1},x_{0,2}) \in \overline{D_S}-D_S$ and that $(x_{0,1},x_{0,2})$ is in exactly two curves from $S_i$ ($i=1,2,3,4$), the rank of the differential of $e_m$ is $2$. 

From Case 1-1-1, Case 1-1-2, and Case 1-1-3, the differential of $e_m$ at each point of $X_m$ is of rank $2$. We give small remark on other objects. We can put $a=t_a$ such that $1<t_a<{(\frac{3}{2})}^2-\frac{1}{2}=\frac{7}{4}$. We can also see $X_m \subset \{t \mid -1 \leq t \leq 1\} \times {\mathbb{R}}^{m+1} \subset \{t \mid -a<t<a\} \times {\mathbb{R}}^{m+1} \subset {\mathbb{R}}^{m+2}$. We can define the function $c_m$ as the restriction of ${\pi}_{m+2,1}$ to $X_m$. The properties (\ref{thm:1.1}, \ref{thm:1.2}) are shown to be enjoyed.
 \\
\ \\
STEP 1-2 The properties (\ref{thm:1.3}) and (\ref{thm:1.4}).

By the construction, each contour of $c_m$ is represented as the intersection of $X_{m}$ and the preimage ${{\pi}_{m+2,2}}^{-1}(\{(t_0,t) \mid a_{t_0,1} \leq t \leq a_{t_0,2}\})$ of a closed interval of the form $\{(t_0,t) \mid a_{t_0,1} \leq t \leq a_{t_0,2}\}$ and compact. In addition, each critical contour of $c_m$ must be represented as the intersection of $X_{m}$ and the preimage ${{\pi}_{m+2,2}}^{-1}(\{(\pm \frac{1}{2},t) \mid a_{\pm \frac{1}{2},1} \leq t \leq a_{\pm \frac{1}{2},2}\})$ of a closed interval of the form $\{(\pm \frac{1}{2},t) \mid a_{\pm \frac{1}{2},1} \leq t \leq a_{\pm \frac{1}{2},2}\}$ and compact. For each critical contour of $c_m$, we can have its small, open and connected neighborhood in $X_m$ which is represented as a union of contours of $c_m$ which are all regular except the given one and whose closure in $X_m$ is also saturated with respect to $c_m$ and compact. In other words, such a union is a {\it saturated} space with respect to $c_m$, where we respect \cite{saeki3}. For each regular contour of $c_m$, we can also have its small, open and connected neighborhood in $X_m$ which is represented as a union of regular contours of $c_m$ and whose closure in $X_m$ is also saturated with respect to $c_m$ and compact. By arguments from \cite{saeki2, saeki3}, $R_{c_m}$ is regarded as the Reeb graph of $c_m$ which is infinite. 
Each edge $e$ of the Reeb graph is represented as the image $q_{c_m}(U_e)$ of a saturated open set $U_e$ with respect to $c_m$ whose closure contains one or two critical contours of $c_m$.
We can choose a suitable graph $G_0$ for the property (\ref{thm:1.3}). The quotient map $q_{c_m};X_m \rightarrow R_{c_m}$ is shown to be proper from the argument and the property (\ref{thm:1.4}) is shown to be enjoyed.  \\

\ \\
This completes the proof of Theorem \ref{thm:1}.
\end{proof}

\section{On Theorem 2.}
We discuss Theorem \ref{thm:2} more precisely, and present this in a revised way, as Theorem \ref{thm:3}.

Hereafter, let ${\mathbb{C}}^k$ denote the {\it $k$-dimensional complex affine space}. Let ${\mathbb{R}}^k$ be identified with the real part $\{x=(x_1, \cdots x_k) \in {\mathbb{R}}^k\} \subset {\mathbb{C}}^k\}$ of ${\mathbb{C}}^k$ canonically.

\begin{Def}
A smooth map $c:{\mathbb{R}}^{k_1} \rightarrow {\mathbb{R}}^{k_2}$ on ${\mathbb{R}}^{k_1}$ is {\it densely real analytic} or {\it D-RAn} if it is real analytic outside a subset $Z$ such that ${\mathbb{R}}^{k_1}-Z$ is dense in ${\mathbb{R}}^{k_1}$.
Furthermore, if the restriction of $c$ to ${\mathbb{R}}^{k_1}-Z$ is the restriction of a complex analytic map defined in the subset $U_{{\mathbb{R}}^k}-Z$ of an open neighborhood $U_{{\mathbb{R}}^k} \supset {\mathbb{R}}^k$ of ${\mathbb{R}}^k$ in ${\mathbb{C}}^k$, then $c$ is said to be {\it densely complex real analytic} or {\it D-CRAn}.
\end{Def}
\begin{Ex}
	\label{ex:1}
	A function $c:\mathbb{R} \rightarrow \mathbb{R}$ such that $c(x):=0$ ($x \leq 0$) and that $c(x):=e^{-\frac{1}{x}}$ ($x \geq 0$) is D-RAn. It is not D-CRAn.
\end{Ex}
\begin{Ex}
	\label{ex:2}
		A function $c:\mathbb{R} \rightarrow \mathbb{R}$ defined by $c(x):=R \times e^{-\frac{1}{x^2}}(\frac{1}{x})^{j_0}{\Sigma}_{j \in J} (\frac{1}{x})^{i_{j}}T_{j}(\sin \frac{1}{x},\cos \frac{1}{x})$ ($x \neq 0$) with the following and $c(0):=0$ is D-CRAn. 
		\begin{itemize}
\item $R \neq 0$ is a real number.
			\item $j_0$ is an integer.
			\item $J$ is a non-empty finite set.
			\item $i_j$ is a non-negative integer defined for each $j \in J$.
			\item $T_j(x)$ is a real polynomial function defined on $\mathbb{R}$ and defined for each $j \in J$.
		\end{itemize}
		The $i$-th derivative of the function $c$ is also represented in this form for each positive integer $i>0$.
	\end{Ex}

	Note that the function of Example \ref{ex:2} is represented in a single elementary function on the dense set $\mathbb{R}-\{0\} \subset \mathbb{R}$. Hereafter, we call this function a {\it special} D-CRAn function or an {\it S-D-CRAn} function. See also \cite[A proof of Theorem 2, Remark 1]{kitazawa8}, where we do not assume related knowledge. Such a function is important in real analysis and it is also presented in \cite[A proof of Theorem 2, Remark 1]{kitazawa8} as a key object.

\begin{Thm}
\label{thm:3}
We can have a case of Theorem \ref{thm:2} with the following.
\begin{enumerate}
\setcounter{enumi}{7}
\item \label{thm:3.1}
$c_{Y,m,i}$ is the restriction of ${\pi}_{m+2,1}$ to $Y_{m,i}$.
\item  \label{thm:3.2} The image of the restriction of ${\pi}_{m+2,2}$ to $Y_{m,i}$ is as follows.
\begin{enumerate}
\item \label{thm:3.2.1} The closure of the region surrounded by a circle $S_{1,R_0}:=\{(x_1,x_2) \mid {x_1}^2+{x_2}^2=R_0\}$ whose radius is $R_0>0$ and
$S_{2,1}:=\{(c_0(x_2),x_2) \mid x_2 \in \mathbb{R}\}$, where $c_0$ is the composition of a real polynomial function $p_{0,R_0}(t):=t(R_0-t)$ on $\mathbb{R}$ with a suitably chosen S-D-CRAn function with a sufficiently small positive number $R>0$ and "$c(x):=R \times e^{-\frac{1}{x^2}} \times {\sin}^2 (\frac{1}{x})$ in Example \ref{ex:2}" , in the case $i=1$. 
\item \label{thm:3.2.2} The closure of the region surrounded by the circle $S_{1,R_0}$ and $S_{2,3}$, the curve obtained by a suitable rotation of $S_{2,1}$ around the origin $0 \in \mathbb{R}$, in the case $i=3$. 
\item \label{thm:3.2.3} The closure of the region surrounded by the circle $S_{1,R_0}$ and $S_{2,2}$, the curve obtained by a suitable parallel transformation with a suitable affine transformation of $S_{2,2}$, in the case $i=2$. 
\end{enumerate}

\end{enumerate}
\end{Thm}

\begin{proof}[A proof of Theorem \ref{thm:3}, with Theorem \ref{thm:2}]
Let $m_1$ and $m_2$ be positive integers. Let $m:=m_1+m_2$. 

We can define $Y_{m,1}:=\{(x_1,x_2,(y_{1,j})_{j=1}^{m_1},(y_{2,j})_{j=1}^{m_2}) \mid R_0-{x_1}^2-{x_2}^2-{\Sigma}_{j=1}^{m_1} {y_{1,j}}^2=0, x_1-c_0(x_2)-{\Sigma}_{j=1}^{m_1} {y_{2,j}}^2=0\}$ and this is regarded as the zero set of a smooth map $e_{Y,m}:{\mathbb{R}}^{m+2} \rightarrow {\mathbb{R}}^2$ defined as in "A proof of Theorem \ref{thm:1}" and an $m$-dimensional smooth compact submanifold of ${\mathbb{R}}^{m+2}$. We can prove this fact as in "A proof of Theorem \ref{thm:1}", by applying implicit function theorem. We omit precise exposition on this.
 
The function $c_0$ is a D-CRAn function whose restriction to $\mathbb{R}-\{0\}$ is real analytic and extends to $\mathbb{C}-\{0\}$ as a holomorphic function in the unique way. In addition, the value of the $j$-th derivative of $c_0$ at $0$ is always $0$. The set $Z_{m,1}$ can be chosen in such a way that $Y_{m,1} \bigcap Z_{m,1}:=\{(0,R_0,0),(0,-R_0,0)\} \subset \mathbb{R} \times \mathbb{R} \times {\mathbb{R}}^m={\mathbb{R}}^{m+2}$ by the structure of the manifolds and the maps. The Reeb space $R_{c_{Y,m,1}}$ of the restriction $c_{Y,m,1}$ of ${\pi}_{m+2,1}$ to $Y_{m,1}$ is not homeomorphic to a graph since around $(0,0,0)$, infinitely many real numbers $x_{2,a}$ can be chosen and yield infinitely many critical points of $c_{Y,m,i}$ of the form $(0,x_{2,a},y_{1,a,j},0) \in {\mathbb{R}}^2 \times {\mathbb{R}}^{m_1} \times {\mathbb{R}}^{m_2}$. Note also that these points yield infinitely many critical contours of $c_{Y,m,1}$. See also \cite[A proof of Theorem 1]{kitazawa8}. Furthermore, by considering the projection of $Y_{m,1}$ to the line $\{0\} \times \mathbb{R} \times \{0\} \subset \mathbb{R} \times \mathbb{R} \times {\mathbb{R}}^m={\mathbb{R}}^{m+2}$, we have a smooth function with exactly two critical points and by a well-known classical theorem and the structure of the manifolds and the maps, this is diffeomorphic to $S^m$. 

This completes the proof for the $c_{Y,m,1}$ case.

We investigate the $c_{Y,m,3}$ case. 

We can have a slightly rotated version of the case $c_{Y,m,1}$ around $\{(0,0,y) \mid y \in {\mathbb{R}}^m\} \subset {\mathbb{R}}^{m+2}$. By the form of manifolds and maps, the number of critical values of $c_{Y,m,3}$ is finite. From \cite{saeki2}, we have the Reeb graph $R_{c_{Y,m,3}}$. The set $Z_{m,3}$ can be chosen in such a way that $Y_{m,3} \bigcap Z_{m,3}$ is a two-point set and this only adds at most two vertices to the Reeb graph.

This completes the proof for the case $c_{Y,m,3}$. 

We investigate the $c_{Y,m,2}$ case. 

We can have a slightly changed version of the case $c_{Y,m,1}$. We consider the case $c_{Y,m,1}$ and $S_{2,1}$ is moved to $S_{2,1,1}$ by a parallel transformation in such a way that $(x_1,x_2) \in S_{2,1}$ is moved to $(x_1-t_1,x_2)$ with a sufficiently small $t_1>0$. After that, we change each point of $S_{2,1}$ which is also in the disk $D_{1,R_0}:=\{(x_1,x_2) \mid {x_1}^2+{x_2}^2 \leq R_0\}$ by the uniquely defined affine transformation mapping $(x_1-t_1,x_2)$ to $(x_1-t_1,t_2x_2)$ and $(-t_1,\pm R_0)$ to $(-t_1,\pm t_2R_0) \in S_{1,R_0}$ with $0<t_2<1$. 
The curve $S_{2,1,1}$ is mapped onto $S_{2,1,2}$. The image of the restriction of ${\pi}_{m+2,2}$ to $Y_{m,1}$ is moved and by the construction, we can have a similar map $c_{Y,m,4}$ on a smooth compact submanifold $Y_{m,4} \subset {\mathbb{R}}^{m+2}$, diffeomorphic to $S^m$, and the set $Z_{m,4}$ like $Z_{m,1}$ and $Z_{m,3}$ is mapped onto $\{(t,t_2R_0) \mid t \in \mathbb{R}\} \sqcup \{(t,-t_2R_0) \mid t \in \mathbb{R}\}$ and the intersection of this and the image of the sphere $Y_{m,4}$ by the projection ${\pi}_{m+2,2}$ is a disjoint union of two straight segments. By considering a rotation as in the case $c_{Y,m,4}$, we have a desired situation for the $c_{Y,m,2}$ case similarly.

 This completes the proof for the case $c_{Y,m,2}$. 

This completes the proof.

\end{proof}
\begin{Rem}
	\label{rem:1}
	In \cite{kitazawa5}, the Reeb graph $R_c$ of a continuous function $c:X \rightarrow \mathbb{R}$ on a closed manifold $X$ which is smooth outside some set $Z$ of Lebesgue measure $0$ is considered and there, a point $v \in R_c$ is a vertex if and only if at least one of the following hold.
	\begin{itemize}
		\item ${q_c}^{-1}(v)$ is a critical contour of $c$. 
		\item ${q_c}^{-1}(v) \bigcap Z$ is non-empty. 
	\end{itemize}
	Note that if $Z$ is not of Lebesgue measure $0$, then $X-Z$ is not dense in $X$.
	Compare this to the definition of the graph $R_{c,Z}$.
\end{Rem}
\begin{Rem}
	\label{rem:2}
	Related to Remark \ref{rem:1}, a continuous map on a compact topological space and its Reeb space is investigated in \cite{gelbukh}. For example, for a smooth function on a closed manifold, the Reeb space is a so-called {\it Peano continuum} and homotopic to a finite graph. See \cite[Theorem 7.5]{gelbukh}.
\end{Rem}

 \section{Conflict of interest and Data availability.}
  \noindent {\bf Conflict of interest.} \\
 The author is a researcher at Osaka Central Advanced Mathematical Institute (OCAMI researcher). This is supported by MEXT Promotion of Distinctive Joint Research Center Program JPMXP0723833165. He thanks this, where he is not employed there. \\
  \ \\
  {\bf Data availability.} \\
  No data other than the present article is generated.

\end{document}